\documentclass{ifacconf}
\usepackage{color}

\usepackage{amsmath, amssymb} 
\usepackage{natbib}

\usepackage{multicol}
\usepackage{nicefrac}
\usepackage{graphicx}
\usepackage{epstopdf}
\usepackage{tikz}
\usepackage{bm}
\usepackage{picture}
\usepackage{epstopdf}
\usepackage{amsmath}
\usepackage{amsfonts}
\usepackage{amssymb}

\allowdisplaybreaks

\begin{document}
\begin{frontmatter}

\title{Prescribed Time Dual-Mode Extremum Seeking Control\thanksref{footnoteinfo}}

\thanks[footnoteinfo]{This work was supported by Mitacs.}

\author[First]{A. Waterman} 
\author[Second]{D. Dochain} 
\author[First]{M. Guay}
\address[First]{Department of Chemical Engineering, Queen's University, Kingston,  ON K7L 3N6, Canada. (email: a.waterman@queensu.ca, martin.guay@chee.queensu.ca}
\address[Second]{Avenue Georges Lemaˆıtre 4-6, B-1348, Louvain-La-Neuve, Belgium (e-mail: denis.dochain@uclouvain.be)}
\begin{abstract}
  We propose a dual mode extremum seeking control design technique that achieves real-time optimization of an unknown measured cost function in a prescribed time. The controller is shown to achieve semi-global practical stability of the optimal equilibrium for the state variables and the input variable for a class of nonlinear dynamical control systems with unknown dynamics. The design technique proposes a timescale transformation that enables the use of dither signals with increasing frequencies. The proposed timescale transformation is designed to avoid the singularity occurring at the prescribed time. A simulation study is performed to demonstrate the effectiveness of the proposed technique.
\end{abstract}


\end{frontmatter}

\section{Introduction} \label{sec0}
Extremum-seeking control (ESC) has a long history \cite{5572972} as a technique that can be used to solve a variety of real-time optimization problems \cite{5717929}. Following the proofs of the stability properties for ESC in \cite{KRSTIC2000595} and \cite{TAN2006889} there has been a substantial amount of research into the limitations and applications of ESC. An appealing feature of ESC is its applicability to model-free optimization. Model-free optimization methods refer to techniques that do not require knowledge of the mathematical relation between the objective function and the decision variables. ESC considers various approximation techniques to estimate the gradient of a measured objective function using persistent input perturbations.  

A number of recent studies have led to the development of extremum seeking control schemes with finite-time and fixed-time practical convergence properties. In \cite{9148026:20a} and \cite{POVEDA20205356:20b}, an ESC design technique was developed to achieve fixed-time practical convergence of a class of static maps to the unknown optimum. An alternative technique was also proposed in \cite{GUAY202131} and \cite{9304072} for the same class of systems where finite-time practical convergence of the averaged system was achieved as perturbations vanish. 

Finite-time convergence methods can generally be classified in the following three categories: finite-time, fixed-time, and prescribed-time. Finite-time methods are characterized by a dependence on initial conditions and design parameters to upper bound the settling time for the system. Fixed-time methods are independent of the initial conditions and conservatively upper-bound the settling time with the design parameters.

Prescribed-time methods do not depend on initial conditions or design parameters, and instead specify the settling time as a control parameter directly. This desirable property reduces the tuning effort of control parameters which can then be used without the need to re-tune for different initial conditions. In \cite{https://doi.org/10.1002/rnc.4084}, the system states are scaled by a function of time that approaches infinity as the settling time is approached. The controller is then designed in terms of the scaled state dynamics. A similar approach is demonstrated in \cite{9867788} for trajectory tracking on a 7-DOF robot manipulator. The prescribed-time convergence is achieved by scaling the system state with a monotonically increasing time-varying function that approaches $\infty$ as the prescribed time is approached. Another method is proposed in \cite{KRISHNAMURTHY2020108860} for a general class of nonlinear systems with system uncertainties. The control design uses a dynamic high gain scaling technique designed for asymptotic  stabilization of nonlinear systems. This is combined with a timescale transformation from the prescribed time horizon to the infinite time horizon, and time-dependent forcing terms in the scaling dynamics. The complete control system is shown to provide prescribed-time convergence of the state and input to the origin. 

There are relatively few ESC schemes with prescribed-time convergence properties in literature. A class of non-smooth ESC with prescribed-time stability was proposed in \cite{9369042} through design of class $\mathcal{K}\mathcal{L}_\mathcal{T}$ functions \cite{RIOS2018103}. In the cases where the system can be described by a static map, it is shown that the algorithm allows the convergence time to be completely prescribed by the user, and in the case of a dynamic plant to a bounded neighborhood of the prescribed time.
 In \cite{9867881:22a} and \cite{9867441:22b}, a ``chirpy" perturbation along with demodulation signals that utilize unbounded increasing frequencies is implemented in place of the standard constant frequency sinusoidal dither signal. This ``chirped" approach combined with a timescale transformation and increasing unbounded tuning gains enable the prescribed-time convergence of the average closed-loop system.


A restriction of ESC is that the optimization process must be slower than the process dynamics. In practice this time-scale separation limits the transient performance of the system. \cite{7386601} and \cite{8430884} demonstrate a dual-mode ESC approach that mitigates the impact of the time-scale separation constraint. The dual-mode controller combines a proportional feedback with an integral term. With appropriate assumptions the controller can be shown to stabilize the unknown dynamical system to the optimal equilibrium for the output, state and input variables. \cite{GUAY202355} provides a dual-mode ESC design with practical finite-time stability of the system to the optimum equilibrium for the state, measured output, and input variables. The proportional feedback component of the dual-mode is used to achieve finite-time stability while the integral feedback corrects for the true value of the input variable at steady-state. 

In this paper, we propose a dual-mode ESC approach that achieves prescribed-time practical convergence of the states and input of a dynamical system to the optimum of a measured cost function. Following an approach inspired by \cite{KRISHNAMURTHY2020108860}, a timescale transformation is proposed with a controller that achieves asymptotic stability in the new timescale and prescribed-time stability in the original timescale.

The paper is structured as follows. The problem formulation is given in Section \ref{sec1}. The proposed ESC and the stability analysis are presented in Section \ref{sec2} and \ref{sec3}. Section \ref{sec4} presents a brief simulation study. Conclusions are presented in Section \ref{sec5}.  

\section{Problem Formulation} \label{sec1} 
We consider a class of unknown nonlinear systems described by the following dynamical system: 
\begin{subequations}  
	\begin{align}
   \dot{x} &=f(x)+g(x)u \label{eq:sysdyn} \\
   y&=h(x)  \label{eq:sysout}
   \end{align}  
   \end{subequations}
where $x \in \mathbb{R}^n$ are the state variables, $u \in \mathbb{R}$ is the input variable, and $y \in \mathbb{R}$ is the measured output variable.  It is assumed that the function $h:\mathbb{R}^n  \rightarrow \mathbb{R}$ is twice continuously differentiable. In addition, the vector valued functions $f:\mathbb{R}^n \rightarrow \mathbb{R}^n$ and $g:\mathbb{R}^n \rightarrow \mathbb{R}^n$ are assumed to be sufficiently smooth. 

The cost function, $h(x)$, meets the following assumption. 
\begin{assum} \label{assum:cost1}
The function $h(x)$ is such that its gradient vanishes only at the minimizer $x^*$, that is:
$$\left.\frac{\partial h}{\partial x}\right|_{x=x^\ast}=0.$$
The Hessian at the minimizer is assumed to be positive and nonzero. In particular, there exists a positive constant $\alpha_h$ such that
$$\frac{\partial^2 h}{\partial x \partial x^\top }  \geq \alpha_h I$$ for all $x \in \mathcal{X} \subset \mathbb{R}$. 
\end{assum}
  
\begin{assum} \label{assum:lgh}
The function $|L_g h|$ is such that:
\begin{align*}
\beta_1 (h-h^\ast)\leq |L_g h|^2 \leq \beta_2 (h-h^\ast).
\end{align*}
for some constants $0<\beta_1<\beta_2$.
\end{assum}

\begin{assum} \label{assum:stab}
Let $D \subset \mathbb{R}^n$ be an open containing $x^\ast$ and $D_u \subset \mathbb{R}$, an open set containing $u^\ast$.For the system \eqref{eq:sysdyn}-\eqref{eq:sysout}, it is assumed that: 
\begin{align}
L_f h  + L_g h u - k |L_gh|^2 \leq -\alpha_1 \| x- \pi(u)\|^2
\end{align}
for $x \in D \subset \mathbb{R}^n$, $u \in D_u \subset \mathbb{R}$ and a positive constant $\alpha_1$. 
\end{assum}

Next, we consider the steady-state behaviour of the system to define the steady-state cost of system \eqref{eq:sysdyn}-\eqref{eq:sysout}. For any fixed $u=\hat{u}$, we assume that there exists a steady-state manifold $x=\pi(\hat{u})$ such that:
\begin{align*}
f(\pi(\hat{u}))+g(\pi(\hat{u}))\hat{u}-k g(\pi(\hat{u}))g(\pi(\hat{u}))^\top \frac{\partial h(\pi(\hat{u}))^\top}{\partial x}=0.
\end{align*}
The steady-state cost function is the value of the cost function $h(x)$ along the steady-state manifold. It is given by:
\begin{align} \label{eq:sscost}
\ell(\hat{u}) = h(\pi(\hat{u})).
\end{align}

We first establish some local properties of the system around the optimum input $x^\ast=\pi(u^\ast)$. Let $\tilde{x}=x-x^\ast$ and $\tilde{u}=u-u^\ast$. 
Let $L_g h (\pi(u))$ denote the value of the Lie derivative $L_g h$ along the steady-state manifold.
\begin{lem} \label{lem1}
Let Assumptions \ref{assum:cost1} and \ref{assum:stab} be satisfied for the nonlinear system \eqref{eq:sysdyn} with cost \eqref{eq:sysout}. Then the steady-state cost \eqref{eq:sscost} is such that
\begin{align*}
(u-u^\ast) L_g h(\pi(u)) \geq \alpha_2 \|u-u^\ast \|^{2}
\end{align*}
for a positive constant $\alpha_2>0$.
\end{lem}
\begin{proof}
By Assumption \ref{assum:stab}
\begin{align*}
L_f h + L_g h u^\ast - k |L_g h|^2 \leq -\alpha_1 \|x- \pi(u^\ast)\|^2
\end{align*}
For points on the steady-state manifold, one obtains:
\begin{align*}
L_f h(\pi(\hat{u})) + L_g h(\pi(\hat{u}))& u^\ast  -k | L_g h(\pi(\hat{u}))|^2 \\ & \leq -\alpha_1 \|\pi(\hat{u})- \pi(u^\ast)\|^2
\end{align*}
Since $L_f h(\pi(\hat{u})) + L_g h(\pi(\hat{u})) \hat{u} - k | L_g h(\pi(\hat{u}))|^2 =0$, we have:
\begin{align*}
-L_g h(\pi(\hat{u})) \hat{u}+ L_g h(\pi(\hat{u})) u^\ast \leq -\alpha_1 \|\pi(\hat{u})- \pi(u^\ast)\|^2
\end{align*}
This leads to 
\begin{align*}
-L_g h(\pi(\hat{u})) (\hat{u}-u^\ast) & \leq -\alpha_1 \|\pi(\hat{u})- \pi(u^\ast)\|^2
\end{align*}
or:
\begin{align*}
L_g h(\pi(\hat{u})) (\hat{u}-u^\ast) - \alpha_1 \|\pi(\hat{u})- \pi(u^\ast)\|^2 & \geq 0.
\end{align*}
Finally:
\begin{align*}
L_g h(\pi(\hat{u})) (\hat{u}-u^\ast)  & \geq   \alpha_1 \|\pi(\hat{u})- \pi(u^\ast)\|^2.
\end{align*}

Hence, the steady-state cost is such that:
\begin{align*}
L_g h(\pi(\hat{u})) (\hat{u}-u^\ast)  & \geq   \alpha_2 \|\hat{u}-u^\ast\|^2.
\end{align*}
For some positive constant. This completes the proof. 
\end{proof}

\begin{rem}
As a result,  the stability conditions of the closed-loop system arising from Assumption \ref{assum:stab} guarantees that the expression $L_g h(\pi(\hat{u})$ plays the role of the gradient of a strictly convex function with the same optimum as $h(x)$.
\end{rem}

The objective of this study is to develop an ESC design technique that guarantees prescribed-time convergence to a neighbourhood of the unknown equilibrium minimizer, $(x^\ast,\, u^\ast)$, of the unknown measured objective function $y=h(x)$.
The proposed target controller is the dual-mode dynamical control given by:
\begin{equation}\label{eq:dm_1}
\begin{aligned}
u =& - k L_gh  + \hat{u} \\
\dot{\hat{u}}=& -\frac{k}{\tau_I} L_gh 
\end{aligned}
\end{equation}
In the next lemma, we state the stability properties of the target average closed-loop system given by:
\begin{equation} \label{eq:tarsys}
\begin{aligned}
\dot{x} =& f(x) - k g(x) L_g h + g(x) \hat{u} \\
\dot{\hat{u}} =& - \frac{k}{\tau_I} L_g h.
\end{aligned}
\end{equation}

\begin{lem} \label{lem2}
Consider system \eqref{eq:tarsys}. Let Assumptions \ref{assum:cost1}, \ref{assum:lgh} and \ref{assum:stab} be met. Then the optimum state $(x^\ast, u^\ast)$ is an exponentially stable equilibrium of the system.
\end{lem}
\begin{proof}
We pose the Lyapunov function candidate: $V= \alpha (h(x)-h(x^\ast)) + \frac{\beta}{2} \tilde{u}^2$. Its rate of change along the trajectories of \eqref{eq:tarsys} is:
\begin{align*}
\dot{V}=\alpha\big(L_f h - k (L_g h)^2 + L_g h \hat{u}) - \frac{\beta k }{\tau_I} L_g h \tilde{u}.
\end{align*}

By Assumptions \ref{assum:stab}, one can write:
\begin{align*}
\dot{V}\leq -\alpha \alpha_1 \| x- \pi(\hat{u}) \|^2 - \frac{\beta k }{\tau_I} L_g h \tilde{u}.
\end{align*}
Next, we consider the expression:
\begin{align*}
\dot{V}\leq -\alpha \alpha_1 &\| x- \pi(\hat{u}) \|^2 - \frac{\beta k }{\tau_I} L_g h(\pi(\hat{u}) \tilde{u} \\
& -\frac{\beta k }{\tau_I} (L_g h - L_g h(\pi(\hat{u})) \tilde{u}.
\end{align*}
By smoothness of $g(x)$ and $h(x)$, there exists a positive constant $L_1$ such that:
\begin{align*}
\dot{V}\leq -\alpha \alpha_1 &\| x- \pi(\hat{u}) \|^2 - \frac{\beta k }{\tau_I} L_g h(\pi(\hat{u})) \tilde{u} \\
& +\frac{\beta k L_1 }{\tau_I}\| x- \pi(\hat{u}) \| \| \tilde{u}\|.
\end{align*}
Using Lemma \eqref{lem1}, we obtain the following inequality:
\begin{align*}
\dot{V}\leq -\alpha \alpha_1 \| x- \pi(\hat{u}) &\|^2 - \frac{\beta k \alpha_2 }{\tau_I} \| \tilde{u} \|^2\\
& +\frac{\beta k L_1 }{\tau_I}\| x- \pi(\hat{u}) \| \| \tilde{u}\|.
\end{align*}
We apply Young's inequality to the last term in the last equality, we can write:
\begin{align*}
\dot{V}\leq -\bigg( \alpha &\alpha_1 -\frac{\beta k L_1 k_1}{2 \tau_I}\bigg) \| x- \pi(\hat{u}) \|^2 
\\ & -\bigg( \frac{\beta k \alpha_2 }{\tau_I}-\frac{\beta k L_1}{2 k_1 \tau_I}\bigg) \| \tilde{u} \|^2.
\end{align*}
for some positive constant $k_1$. We let $k_1= \frac{L_1}{\alpha_2}$ to obtain:
\begin{align*}
\dot{V}\leq -\bigg( \alpha &\alpha_1 -\frac{\beta k L_1^2}{2 \alpha_2 \tau_I}\bigg) \| x- \pi(\hat{u}) \|^2 -\frac{\beta k \alpha_2 }{2 \tau_I} \| \tilde{u} \|^2.
\end{align*}
One can always pick $\alpha$ large enough to guarantee that $$\bigg( \alpha \alpha_1 -\frac{\beta k L_1^2}{2 \alpha_2 \tau_I}\bigg)>0.$$ As a result, one can use the triangle inequality on the first term to obtain:
\begin{align*}
\dot{V}\leq& -\bigg(  \alpha \alpha_1 -\frac{\beta k L_1^2}{2 \alpha_2 \tau_I}\bigg) \| x- \pi(u^\ast) \|^2\\
&+\bigg(  \alpha \alpha_1 -\frac{\beta k L_1^2}{2 \alpha_2 \tau_I}\bigg) \| \pi(\hat{u})- \pi(u^\ast) \|^2 -\frac{\beta k \alpha_2 }{2 \tau_I} \| \tilde{u} \|^2.
\end{align*}
Since $g(x)$ and $f(x)$ are smooth, then it follows that the steady-state map $\pi(u)$ is a Lipschitz function. As a result, there exists a positive constant $L_{\pi}$ such that:
\begin{align*}
\dot{V}\leq& -\bigg(  \alpha \alpha_1 -\frac{\beta k L_1^2}{2 \alpha_2 \tau_I}\bigg) \| x- \pi(u^\ast) \|^2\\
&+L_\pi \bigg(  \alpha \alpha_1 -\frac{\beta k L_1^2}{2 \alpha_2 \tau_I}\bigg) \| \tilde{u} \|^2 -\frac{\beta k \alpha_2 }{2 \tau_I} \| \tilde{u} \|^2.
\end{align*}
We let $\alpha=k_2 + \frac{\beta k L_1^2}{2 \alpha_1 \alpha_2 \tau_I}$ to obtain:
\begin{align*}
\dot{V}\leq& -k_2 \alpha_1 \| x- \pi(u^\ast) \|^2 -\bigg(\frac{\beta k \alpha_2 }{2 \tau_I}-k_2 \alpha_1 L_\pi\bigg) \| \tilde{u} \|^2.
\end{align*}
We can finally assign $\beta=\frac{4 k_2 \tau_I  \alpha_1 L_\pi}{\alpha_2k}$ which yields the following inequality:
\begin{align*}
\dot{V}\leq& -k_2 \alpha_1 \| x- \pi(u^\ast) \|^2 -k_2 \alpha_1 L_\pi \| \tilde{u} \|^2.
\end{align*}
As a result, it follows that the optimum operating point $(x^\ast, u^\ast)$ is an exponentially stable equilibrium of the target system \eqref{eq:tarsys}. This completes the proof.
\end{proof}
\section{Dual-mode controller} \label{sec2}
\subsection{Timescale transformation}

Following the strategy proposed in \cite{KRISHNAMURTHY2020108860}, we rewrite the dynamic feedback controller in the $t$ timescale as follows. The time scale $t(\tau) = \frac{T \tau}{T+\tau}$ has a well defined inverse given by:
\begin{align*}
\tau(t) = \frac{t T}{T-t}.
\end{align*} 
This yields the timescale transformation:
\begin{align*}
\frac{d \tau}{dt} = \frac{T^2}{(T-t)^2}, \qquad \tau(0)=0.
\end{align*}
Its inverse defines the timescale transformation
\begin{align*}
    \frac{dt}{d\tau}=\frac{T^2}{(T+\tau)^2}, \qquad t(0)=0.
\end{align*}

The basic strategy for the prescribed time stabilization problem is to design a control system that achieves asymptotic stability in the $\tau$ timescale. The translation of the controller in the $t$ timescale ensures that the system reaches the origin at $t=T$. In the current context, we seek a control system that can reach the optimum equilibrium $(x^\ast,\,u^\ast)$ at a prescribed time $T$.

\subsection{Target prescribed-time controller}

We propose the dynamic feedback:
\begin{equation} \label{eq:PTtarsys}
\begin{aligned}
\frac{d \hat{u}}{dt} =& -\frac{k}{\tau_I}  \frac{T^2}{(T-t)^2} L_g h \\
u =& -k\bigg(1+ \frac{T^2}{(T-t)^2}\bigg)  L_g h + \hat{u}
\end{aligned}
\end{equation}
which is well defined for $t \in [0, T)$.  The following theorem provides the stability properties of the proposed control system.

\begin{lem} \label{lem3}
Consider the nonlinear system \eqref{eq:sysdyn} with cost function \eqref{eq:sysout} subject to the controller \eqref{eq:PTtarsys}. Let Assumptions \eqref{assum:cost1}, \eqref{assum:lgh} and \eqref{assum:stab} be met. Then, for some positive constant $T$, the trajectories of the system are bounded and
\begin{align*} 
\lim_{t \rightarrow T} x(t)=x^\ast, \,\,\, \lim_{t\rightarrow T} \hat{u}(t) = u^\ast.
\end{align*}
where $x^\ast, \, u^\ast$ are the equilibrium conditions that minimize the cost function $h(x)$.
\end{lem}
\begin{proof}
We pose the candidate Lyapunov function $W=\frac{\alpha}{v(t)}(h(x)-h(x^\ast))+\frac{\beta}{2} \tilde{u}^2$ where $v(t)=\frac{(T-t)^2}{T^2}$.  We note that $W\geq \alpha \alpha_h \|x-x^\ast\|^2 + \frac{\beta}{4} \tilde{u}^2$.  Differentiating with respect to $t$, we obtain:
\begin{align*}
\dot{W} =&-\frac{\dot{v}}{v(t)^2}\alpha (h(x)-h(x^\ast)) - \frac{\alpha k}{v(t)}\frac{T^2}{(T-t)^2} L_gh^2 \\
             &+ \frac{\alpha}{v(t)}\bigg( L_f h - k (L_g h)^2 + L_g h \hat{u}\bigg) - \frac{\beta k }{v(t) \tau_I} L_g h(\pi(\hat{u}) \tilde{u} \\
             &-\frac{\beta k }{\tau_I} (L_g h - L_g h(\pi(\hat{u})) \tilde{u}.
\end{align*}

By Assumption \ref{assum:stab}, one can write:
\begin{align*}
\dot{W} =&-\frac{\dot{v}}{v(t)^2}\alpha (h(x)-h(x^\ast)) - \frac{\alpha k}{v(t)}\frac{T^2}{(T-t)^2} L_gh^2 \\
             &- \frac{\alpha \alpha_1}{v(t)} \| x- \pi(\hat{u}) \|^2  - \frac{\beta k}{v(t) \tau_I} L_g h \tilde{u}.
\end{align*}
By smoothness of $g(x)$ and $h(x)$, there exists a positive constant $L_1$ such that:
\begin{align*}
\dot{W} =&-\frac{\dot{v}}{v(t)^2}\alpha (h(x)-h(x^\ast)) - \frac{\alpha k}{v(t)}\frac{T^2}{(T-t)^2} L_gh^2 \\
             &- \frac{\alpha \alpha_1}{v(t)} \| x- \pi(\hat{u}) \|^2 - \frac{\beta k }{v(t) \tau_I} L_g h(\pi(\hat{u}) \tilde{u} \\
             &+\frac{\beta L_1 k }{v(t)\tau_I} \| x- \pi(\hat{u}) \| \| \tilde{u}\|.
            \end{align*}
 Using Lemma \eqref{lem1}, we obtain the following inequality:
\begin{align*}
\dot{W} =&-\frac{\dot{v}}{v(t)^2}\alpha (h(x)-h(x^\ast)) - \frac{\alpha k}{v(t)}\frac{T^2}{(T-t)^2} L_gh^2 \\
             &- \frac{\alpha \alpha_1}{v(t)} \| x- \pi(\hat{u}) \|^2 - \frac{\beta k \alpha_2 }{v(t) \tau_I} \| \tilde{u} \|^2 \\
             &+\frac{\beta L_1 k }{v(t)\tau_I} \| x- \pi(\hat{u}) \| \| \tilde{u}\|.
            \end{align*}
We apply Young's inequality to the last term in the last equality, we can write:
\begin{align*}
\dot{W} =&-\frac{\dot{v}}{v(t)^2}\alpha (h(x)-h(x^\ast)) - \frac{\alpha k}{v(t)}\frac{T^2}{(T-t)^2} L_gh^2 \\
             &-\frac{1}{v(t)}\bigg( \alpha \alpha_1 -\frac{\beta k L_1 k_1}{2 \tau_I}\bigg) \| x- \pi(\hat{u}) \|^2 
\\ & -\frac{1}{v(t)}\bigg( \frac{\beta k \alpha_2 }{\tau_I}-\frac{\beta k L_1}{2 k_1 \tau_I}\bigg) \| \tilde{u} \|^2           \end{align*}
for some positive constant $k_1$. We let $k_1= \frac{L_1}{\alpha_2}$ to obtain:
\begin{align*}
\dot{W} =&-\frac{\dot{v}}{v(t)^2}\alpha (h(x)-h(x^\ast)) - \frac{\alpha k}{v(t)}\frac{T^2}{(T-t)^2} L_gh^2 \\
             &-\frac{1}{v(t)}\bigg( \alpha \alpha_1 -\frac{\beta k L_1^2}{2 \alpha_2 \tau_I}\bigg) \| x- \pi(\hat{u}) \|^2 \\
             &-\frac{\beta k \alpha_2 }{2 v(t) \tau_I} \| \tilde{u} \|^2.           \end{align*}
One can always pick $\alpha$ large enough to guarantee that $$\bigg( \alpha \alpha_1 -\frac{\beta k L_1^2}{2 \alpha_2 \tau_I}\bigg)>0.$$ As a result, one can use the triangle inequality to obtain:
\begin{align*}
\dot{W} =&-\frac{\dot{v}}{v(t)^2}\alpha (h(x)-h(x^\ast)) - \frac{\alpha k}{v(t)}\frac{T^2}{(T-t)^2} L_gh^2 \\
             &-\frac{1}{v(t)}\bigg(  \alpha \alpha_1 -\frac{\beta k L_1^2}{2 \alpha_2 \tau_I}\bigg) \| x- \pi(u^\ast) \|^2\\
&+\frac{1}{v(t)} \bigg(  \alpha \alpha_1 -\frac{\beta k L_1^2}{2 \alpha_2 \tau_I}\bigg) \| \pi(\hat{u})- \pi(u^\ast) \|^2 -\frac{\beta k \alpha_2 }{2  v(t)\tau_I} \| \tilde{u} \|^2.       
\end{align*}
Since $g(x)$ and $f(x)$ are smooth, then it follows that the steady-state map $\pi(u)$ is a Lipschitz function. As a result, there exists a positive constant $L_{\pi}$ such that:
\begin{align*}
\dot{W} =&-\frac{\dot{v}}{v(t)^2}\alpha (h(x)-h(x^\ast)) - \frac{\alpha k}{v(t)}\frac{T^2}{(T-t)^2} L_gh^2 \\
             &-\frac{1}{v(t)}-\bigg(  \alpha \alpha_1 -\frac{\beta k L_1^2}{2 v(t) \alpha_2 \tau_I}\bigg) \| x- \pi(u^\ast) \|^2\\
&+\frac{L_\pi}{v(t)} \bigg(  \alpha \alpha_1 -\frac{\beta k L_1^2}{2 \alpha_2 \tau_I}\bigg) \| \tilde{u} \|^2 -\frac{\beta k \alpha_2 }{2 v(t) \tau_I} \| \tilde{u} \|^2.  
\end{align*}
We let $\alpha=k_2 + \frac{\beta k L_1^2}{2 \alpha_1 \alpha_2 \tau_I}$ to obtain:
\begin{align*}
\dot{W} =&-\frac{\dot{v}}{v(t)^2}\alpha (h(x)-h(x^\ast)) - \frac{\alpha k}{v(t)}\frac{T^2}{(T-t)^2} L_gh^2 \\
             &-\frac{k_2 \alpha_1}{v(t)} \| x- \pi(u^\ast) \|^2 -\frac{1}{v(t)}\bigg(\frac{\beta k \alpha_2 }{2 \tau_I}-k_2 \alpha_1 L_\pi\bigg) \| \tilde{u} \|^2
\end{align*}
We can finally assign $\beta=\frac{4 k_2 \tau_I  \alpha_1 L_\pi}{\alpha_2k}$ which yields the following inequality:
\begin{align*}
\dot{W} =&-\frac{\dot{v}}{v(t)^2}\alpha (h(x)-h(x^\ast)) - \frac{\alpha k}{v(t)}\frac{T^2}{(T-t)^2} L_gh^2 \\
             &-\frac{k_2 \alpha_1}{v(t)} \| x- \pi(u^\ast) \|^2 -\frac{k_2 \alpha_1 L_\pi}{v(t)} \| \tilde{u} \|^2.
\end{align*}
Next, we invoke the definition of $v(t)$, and, Assumption \ref{assum:lgh}, to upper bound the second term of the last inequality, one obtains:
\begin{align*}
\dot{W} \leq&\frac{2T^2}{(T-t)^3} \alpha (h(x)-h(x^\ast)) - \alpha k \beta_1 \frac{T^4}{(T-t)^4} (h(x)-h(x^\ast)) \\
             &-k_2 \frac{1}{v(t)}\bigg(\alpha_1 \| x- \pi(u^\ast)\|^2 +\alpha_1 L_\pi \tilde{u}^2\bigg).
\end{align*}
Upon substitution of the Lyapunov function $W$, one can write:
\begin{align*}
\dot{W} \leq&\frac{2}{(T-t)} W -  k \beta_1 \frac{T^2}{(T-t)^2} W +\frac{k \beta_1 \beta}{2} \frac{T^2}{(T-t)^2} \tilde{u}^2\\
             &- \frac{k_2}{v(t)}\bigg(\alpha_1 \| x- \pi(u^\ast)\|^2 +\alpha_1 L_\pi \tilde{u}^2\bigg).
\end{align*}
We can assign $k_2=\frac{k \beta_1 \beta}{2 \alpha_1 L_\pi}$ to obtain the following inequality:
\begin{align*}
\dot{W} \leq&\frac{2}{(T-t)} W -  k \beta_1 \frac{T^2}{(T-t)^2} W.
\end{align*}
By integration, this yields:
\begin{align*}
W(t) \leq W(0) \frac{T^2}{(T-t)^2} \exp\bigg(-k \beta_1 \frac{Tt}{(T-t)}\bigg).
\end{align*}
Since $\lim_{t\rightarrow T} W(t) =0$, it follows that $lim_{t \rightarrow T} x(t)=x^\ast$ and $\lim_{t \rightarrow T} \hat{u}=u^\ast$. Furthermore, it can also be shown that $W(t)$ is strictly decreasing over the interval $[0,T)$ for any $W(0)$ if one selects $k\beta_1 T \geq 2$. This shows that one achieves prescribed-time stability of $x^\ast$ and $u^\ast$ which completes the proof.
\end{proof}

\subsection{Dual-mode Prescribed-time ESC Design}

Next, we propose a model free formulation of the proposed dual-mode target controller from the previous section.  The main task is to provide a suitable estimate of the ``gradient-like" term $L_g h$. 

The first element is the addition of a dither signal is added to the input $u$. The input is then given by:
\[u = u_c+ A\sin(\omega \tau) \]
where $A$ and $\omega$ are positive constants and $u_c$ is the feedback control action.  It is important to note that the dither signal is required to be in the $\tau$-timescale. In the $t$-timescale, the input can be written as:
\begin{align*}
u = u_c+ A\sin\bigg(\omega \frac{Tt}{T-t}\bigg).
\end{align*}
where $T>0$ is the prescribed-time.

The estimate of $L_gh$ is then obtained using an extremum seeking control approach. We let the estimation of $L_gh$, denoted by $\xi$, be the output of the low-pass filter in the $t$-timescale:
\begin{align*}
\frac{d\xi}{dt} = -\omega_l \frac{T^2}{(T-t)^2}\bigg(\xi - \frac{2}{A}\sin(\omega\tau)\frac{dy}{dt}\bigg)
\end{align*}
where $\omega_l$ is the bandwidth of the filter. In general, the term $\frac{dy}{dt}$ is replaced by an estimate. The preferred approach is to use a high pass filter. In the prescribed-time framework, the high pass filter is first posed in the slow $\tau$ timescale. This gives:
\begin{align*}
    \frac{d\eta}{d\tau} &= -(\omega_h\eta -y(\tau)) \\
    \nu(\tau) &= -\omega_h^2\eta + \omega_h y(\tau)
\end{align*}
where $\omega_h$ is a positive constant.
The dynamics of the derivative estimate $\nu(\tau)$ are given by:
\begin{align*}
\frac{d\nu}{d\tau} = -\omega_h \bigg(\nu(\tau) - \frac{dy}{d\tau}\bigg).
\end{align*}
Thus, the error signal $\tilde{\nu} = \nu(\tau) - \frac{dy}{d\tau}$ has dynamics described by:
\begin{align*}
\frac{d\tilde{\nu}}{d\tau} = -\omega_h \Tilde{\nu} + \frac{d^2y}{d\tau^2}.
\end{align*}
We can then  write the estimate in the original time scale as:
\begin{align*}
    \frac{d\eta}{dt} &= -\frac{T^2}{(T-t)^2} (\omega_h\eta -y(t)) \\
    \nu(t) &= -\omega_h^2\eta + \omega_h y(t)
\end{align*}
Correspondingly, the dynamics of $\nu$ can be written as:
\begin{align*}
\frac{d\nu}{dt} = -\omega_h \frac{T^2}{(T-t)^2}  \bigg(\nu(t) - \frac{dy}{d\tau}\bigg).
\end{align*}

The estimation of $\frac{dy}{dt}$ is written in the original time scale as:
\begin{align*}
\bar{\nu}(t)=\frac{T^2}{(T-t)^2} \nu(t) 
\end{align*}
which yields:
\begin{align*}
\bar{\nu}(t)= \frac{T^2}{(T-t)^2} \tilde{\nu}(t)  + \frac{T^2}{(T-t)^2} \frac{d y}{d\tau}
\end{align*}
or,
\begin{align*}
\bar{\nu}(t)= \frac{T^2}{(T-t)^2} \tilde{\nu}(t)  + \frac{d y}{dt}
\end{align*}

The estimate for $L_gh$ can be expressed as follows:
\begin{equation} \label{eq-4}
    \frac{d\xi}{dt} = -\omega_l\frac{T^2}{(T-t)^2} \bigg(\xi - \frac{2}{A}\sin(\omega\tau)\bar{\nu}\bigg). 
 \end{equation}
Given the estimate $\xi$, we can pose the proposed dual-mode controller as:
\begin{equation} \label{eq:escc}
\begin{aligned}
\dot{\hat{u}} = &-\frac{k}{\tau_I}\bigg( \frac{T^2}{(T-t)^2}\bigg) \xi \\
    u =& -k\bigg(1+ \frac{T^2}{(T-t)^2}\bigg) \xi + \hat{u} + A\sin{\omega\tau}.
    \end{aligned}
 \end{equation}

\section{Stability Analysis} \label{sec3}

The analysis of stability proceeds in two steps. First, we generate the averaged closed-loop system. The averaged system is generated by integrating the time varying dither signal terms over one period of the sinusoidal. In this case, the period of oscillation is $T=\frac{2 \pi}{\omega}$ in the $\tau$-time scale. This yields the following closed-loop system:
\begin{equation} \label{eq:avgsysPT}
	\begin{aligned} 	
		\dot{x}^a=&f(x^a) + g(x^a) u^a \\
		\dot{\xi}^a= & -\omega_l\bigg( \frac{T^2}{(T-t)^2}\bigg) \bigg(\xi^a - L_g h^a\bigg) \\
		\dot{\hat{u}}^a = &-\frac{k}{\tau_I} \bigg( \frac{T^2}{(T-t)^2}\bigg) \xi^a\\
		u^a = & -k\bigg(1+\frac{T^2}{(T-t)^2}\bigg) \xi^a + \hat{u}^a.
	\end{aligned}
\end{equation}

%

Before we proceed with the analysis, we consider one additional assumption concerning the closed-loop dynamics of $L_gh$ given by:
\begin{align} \label{eq:lghdyn}
\frac{d L_gh}{dt} = L_f L_g h - k L_g^2 h L_g h+ L_g^2 h \hat{u}.
\end{align}
\begin{assum} \label{assum:lghdyn}
The closed-loop dynamics \eqref{eq:lghdyn} are such that:
\begin{itemize}
\item{i)} the iterated Lie derivative term $L_g^2h$ is such that:
\begin{align*}
\beta_3 \leq L_g^2h \leq \beta_4,
\end{align*}
for some positive constants $0<\beta_3<\beta_4$.
\item{ii)} there exists a positive constant $L_3$ such that:
\begin{align*}
|L_f L_g h - k L_g^2 h L_g h+ L_g^2 h \hat{u} | \leq L_3 \| x- \pi(u^\ast)\| + L_4 \|\tilde{u}\|.
\end{align*}
\end{itemize}
\end{assum}

\begin{lem} \label{lem4}
Consider the nonlinear closed-loop system \eqref{eq:avgsysPT}. Let Assumptions \ref{assum:cost1}, \ref{assum:lgh}, \ref{assum:stab} and \ref{assum:lghdyn} be met. Then, for some positive constant $T$, the trajectories of the system are bounded and
\begin{align*} 
\lim_{t \rightarrow T} x(t)=x^\ast, \,\,\, \lim_{t\rightarrow T} \hat{u}(t) = u^\ast.
\end{align*}
where $x^\ast, \, u^\ast$ are the equilibrium conditions that minimize the cost function $h(x)$.
\end{lem}
\begin{proof}
We proceed to the analysis of the prescribed-time average system \eqref{eq:avgsysPT}. Following the proof of Theorem \ref{theo1}, we pose the candidate Lyapunov function:
\begin{align*}
W=\frac{\alpha}{v(t)} (h(x)-h(x^\ast))+ \frac{\beta}{2} \tilde{u}^2 + \frac{\gamma}{2 v(t)} (\xi - L_g h)^2. 
\end{align*}
Upon differentiation with respect to $t$, we get:
\begin{align*}
\dot{W}=&-\frac{\alpha \dot{v}}{v(t)^2} (h(x)-h(x^\ast))- \frac{\gamma \dot{v}}{2v(t)^2} (\xi - L_g h)^2 \\&+ \frac{\alpha}{v(t)} \bigg( L_f h - k\left(1+\frac{1}{v(t)}\right) L_gh \xi + L_g h \hat{u} \bigg)\\
&- \frac{\beta}{v(t)} \xi \tilde{u} -\frac{\gamma \omega_l}{v(t)^2} (\xi - L_g h)^2 - \frac{1}{v(t)}\gamma(\xi - L_g h) \frac{d L_g h}{dt}.
\end{align*}
This can be rewritten as follows:
\begin{align*}
\dot{W}=&-\frac{\alpha \dot{v}}{v(t)^2} (h(x)-h(x^\ast))- \frac{\gamma \dot{v}}{2 v(t)^2} (\xi - L_g h)^2  \\&+ \frac{\alpha}{v(t)} \bigg( L_f h - k L_gh^2 + L_g h \hat{u} \bigg)- \frac{\beta k}{\tau_I v(t)} L_g h\tilde{u} \\
& - \frac{\alpha k}{v(t)}\left(1+\frac{1}{v(t)}\right) L_g h (\xi-L_g h) -  \frac{\alpha k}{v(t)^2} L_gh^2 \\
&- \frac{\beta k}{\tau_I v(t)} (\xi-L_gh) \tilde{u}  -\frac{\gamma \omega_l}{v(t)^2} (\xi - L_g h)^2 \\&-\frac{\gamma}{v(t)}(\xi - L_g h) \bigg( L_f L_g h -k L_g^2 h L_g h + L_g^2 h \hat{u}\bigg) \\
&+\frac{\gamma k}{v(t)^2} (\xi-L_gh) L_g^2 h L_g h \\&+ \frac{\gamma k }{v(t)} \bigg(1+\frac{1}{v(t)}\bigg) L_g^2 h (\xi-L_g h)^2.
\end{align*}
Now, we can write:
\begin{align*}
-\frac{\alpha k}{v(t)} \left(1+\frac{1}{v(t)}\right) L_g h (\xi-L_gh) \leq \frac{2 \alpha k}{v(t)^2}|L_gh||\xi-L_gh| 
\end{align*}
where we used the fact that $\frac{1}{v(t)}\geq 1$ for $t\in[0,T)$.  Similarly, we have:
\begin{align*}
\frac{ \gamma k}{v(t)} \bigg(1+\frac{1}{v(t)}\bigg) L_g^2 h (\xi-L_g h)^2 \leq \frac{2 k \gamma}{v(t)^2} L_g^2 h (\xi-L_gh)^2.
\end{align*}
This yields the following inequality:
\begin{align*}
\dot{W}\leq&-\frac{\alpha \dot{v}}{v(t)^2} (h(x)-h(x^\ast))- \frac{\gamma \dot{v}}{2 v(t)^2} (\xi - L_g h)^2  \\
&+ \frac{\alpha}{v(t)} \bigg( L_f h - k L_gh^2 + L_g h \hat{u} \bigg)- \frac{\beta k}{\tau_I v(t)} L_g h\tilde{u} \\
& + \frac{2 \alpha k }{v(t)^2}| L_g h| |\xi-L_g h| -  \frac{\alpha k}{v(t)^2} L_gh^2 \\
&- \frac{\beta k}{\tau_I v(t)} (\xi-L_gh) \tilde{u}  -\frac{\gamma \omega_l}{v(t)^2} (\xi - L_g h)^2 \\
&-\frac{\gamma}{v(t)}(\xi - L_g h) \bigg( L_f L_g h -k L_g^2 h L_g h + L_g^2 h \hat{u}\bigg) \\
&+\frac{\gamma k}{v(t)^2} (\xi-L_gh) L_g^2 h L_g h \\&+  \frac{2\gamma k}{v(t)^2} L_g^2 h (\xi-L_g h)^2.
\end{align*}
Proceeding as in Lemma \ref{lem3}, we can write:
\begin{align*}
\dot{W}\leq&-\frac{\alpha \dot{v}}{v(t)^2} (h(x)-h(x^\ast))- \frac{\gamma \dot{v}}{2 v(t)^2} (\xi - L_g h)^2  \\
&+ \frac{1}{v(t)} \bigg(\alpha \alpha_1 - \frac{\beta k L_2^2}{2 \alpha_2 \tau_I}\bigg) \| x- \pi(u^\ast)\|^2\\
&+ \frac{L_\pi}{v(t)} \bigg( \alpha \alpha_1 - \frac{\beta k L_2^2}{2 \alpha_2 \tau_I}\bigg) \| \tilde{u}\|^2 - \frac{\beta k \alpha_2}{2 \tau_I v(t)} \|\tilde{u}\|^2 \\
& + \frac{2 \alpha k}{v(t)^2} | L_g h| |\xi-L_g h| -  \frac{\alpha k}{v(t)^2} L_gh^2 \\
&- \frac{\beta k}{\tau_I v(t)} (\xi-L_gh) \tilde{u}  -\frac{\gamma \omega_l}{v(t)^2} (\xi - L_g h)^2 \\
&-\frac{\gamma}{v(t)}(\xi - L_g h) \bigg( L_f L_g h -k L_g^2 h L_g h + L_g^2 h \hat{u}\bigg) \\
&+\frac{\gamma k }{v(t)^2}  (\xi-L_gh) L_g^2 h L_g h \\&+  \frac{2\gamma k}{v(t)^2} L_g^2 h (\xi-L_g h)^2.
\end{align*}
We apply Young's inequality to the indeterminate terms to obtain the following inequality:
\begin{align*}
\dot{W}\leq&-\frac{\alpha \dot{v}}{v(t)^2} (h(x)-h(x^\ast))- \frac{\gamma \dot{v}}{2 v(t)^2} (\xi - L_g h)^2  \\&+ \frac{1}{v(t)} \bigg(\alpha \alpha_1 - \frac{\beta k L_2^2}{2 \alpha_2 \tau_I}\bigg) \| x- \pi(u^\ast)\|^2\\&
+ \frac{L_\pi}{v(t)} \bigg( \alpha \alpha_1 - \frac{\beta k L_2^2}{2 \alpha_2 \tau_I}\bigg) \| \tilde{u}\|^2 - \frac{\beta k \alpha_2}{2 \tau_I v(t)} \|\tilde{u}\|^2 \\
& -\frac{\gamma \omega_l}{v(t)^2} (\xi - L_g h)^2  -  \frac{\alpha}{v(t)^2} L_gh^2 \\
&+ \frac{\alpha k}{k_8 v(t)^2} L_g h^2  + \frac{\alpha k k_8}{v(t)^2} (\xi-L_g h)^2 \\
&+ \frac{\beta k}{2\tau_I  k_9v(t)} \tilde{u} + \frac{\beta k k_9}{2 \tau_I v(t)} (\xi-L_gh)^2   \\
& +\frac{\gamma  k_{10}  L_3 }{2v(t)}(\xi - L_g h)^2+\frac{\gamma L_3}{2 k_{10} v(t)}\| x-\pi(u^\ast) \|^2 \\
&+\frac{  \gamma k_{11} L_4 }{2v(t)}(\xi - L_g h)^2+\frac{\gamma L_4}{2 k_{11} v(t)} \tilde{u}^2 \\
&+\frac{\gamma k  \beta_4 k_{12} }{2 v(t)^2} (\xi-L_gh)^2 + \frac{\gamma k \beta_4}{2 k_{12} v(t)^2}L_g h^2 \\
&+  \frac{2\gamma k \beta_4}{v(t)^2} (\xi-L_g h)^2.
\end{align*}
for positive constants $k_8$, $k_9$, $k_{10}$, $k_{11}$ and $k_{12}$.  Following the proof of Lemma \ref{lem1}, we can write the following after collecting similar terms:
\begin{align*}
\dot{W}\leq&-\frac{\alpha \dot{v}}{v(t)^2} (h(x)-h(x^\ast))- \frac{\gamma \dot{v}}{2 v(t)^2} (\xi - L_g h)^2  \\
&- \bigg(\frac{k_2 \alpha_1}{v(t)}-\frac{\gamma L_3}{2 k_{10} v(t)}\bigg) \| x- \pi(u^\ast)\|^2\\
& - \bigg( \frac{L_\pi k_2 \alpha_1}{v(t)} - \frac{\beta k}{2 \tau_I k_9v(t)} - \frac{\gamma L_4}{2 k_{11} v(t)}\bigg)  \tilde{u}^2\\
&-\bigg(  \frac{\alpha k}{v(t)^2} -\frac{\alpha k}{k_8 v(t)^2}-\frac{\gamma k \beta_4}{2 k_{12} v(t)^2} \bigg)  L_gh^2 \\
& -\bigg( \frac{\gamma \omega_l}{v(t)^2}-\frac{\alpha k k_8}{v(t)^2}- \frac{\beta k k_9}{2 \tau_I v(t)} -\frac{  \gamma L_3 k_{10} }{2v(t)} \\&\hspace{0.5in}-\frac{  \gamma L_4 k_{11} }{2v(t)} -\frac{\gamma k  \beta_4 k_{12} }{2 v(t)^2} -\frac{2\gamma k \beta_4}{v(t)^2}  \bigg)  (\xi-L_g h)^2.
\end{align*}
Since $\frac{1}{v(t)}\geq 1$, this can be written as:
\begin{align*}
\dot{W}\leq&-\frac{\alpha \dot{v}}{v(t)^2} (h(x)-h(x^\ast))- \frac{\gamma \dot{v}}{2 v(t)^2} (\xi - L_g h)^2  \\&- \frac{1}{v(t)}\bigg(k_2 \alpha_1-\frac{\gamma L_3}{2 k_{10}}\bigg) \| x- \pi(u^\ast)\|^2\\& - \frac{1}{v(t)}\bigg( L_\pi k_2 \alpha_1 - \frac{\beta k}{2 \tau_I k_9} - \frac{\gamma L_4}{2 k_{11}}\bigg)  \tilde{u}^2\\
&-\frac{1}{v(t)^2}\bigg(  \alpha k -\frac{\alpha k}{k_8}-\frac{\gamma k \beta_4}{2 k_{12}} \bigg)  L_gh^2 \\
& -\frac{1}{v(t)^2}\bigg( \gamma \omega_l-\alpha k k_8- \frac{\beta k k_9}{2\tau_I} -\frac{ \gamma L_3 k_{10}  }{2} \\
&\hspace{0.5in}-\frac{ \gamma L_4k_{11}  }{2} -\frac{\gamma k  \beta_4 k_{12} }{2 } -2 \gamma k \beta_4  \bigg)  (\xi-L_g h)^2.
\end{align*}
We let $k_8=4$, $k_9 = \frac{2 \beta}{\tau_I  L_\pi \alpha_1 k_2}$, $k_{10} = \frac{\gamma L_3}{k_2 \alpha_1}$, $k_{11}=\frac{2 \gamma L_4}{L_\pi k_2 \alpha_1}$ and $k_{12}=\frac{2 \gamma \beta_4}{\alpha}$. We also assign:
\begin{align*}
w_l = k_{13} + \frac{1}{\gamma} \bigg( k \alpha k_8&+ \frac{\beta k k_9}{2 \tau_I} +\frac{ k_{10} \gamma L_3 }{2}\\&+\frac{ k_{11} \gamma L_4 }{2} +\frac{\gamma k  \beta_4 k_{12} }{2 } -2\beta_4 k\gamma  \bigg). 
\end{align*}
 This yields:
\begin{align*}
\dot{W}\leq&-\frac{\alpha \dot{v}}{v(t)^2} (h(x)-h(x^\ast))- \frac{\gamma \dot{v}}{2 v(t)^2} (\xi - L_g h)^2  \\&- \frac{k_2 \alpha_1}{2v(t)}\| x- \pi(u^\ast)\|^2 - \frac{L_\pi k_2 \alpha_1}{4v(t)} \tilde{u}^2\\
&-\frac{\alpha k}{4v(t)^2} L_gh^2 -\frac{\gamma k_{13}}{v(t)^2}  (\xi-L_g h)^2.
\end{align*}
By Assumption \ref{assum:lgh}, we finally obtain:
\begin{align*}
\dot{W}\leq&-\frac{\alpha \dot{v}}{v(t)^2} (h(x)-h(x^\ast))- \frac{\gamma \dot{v}}{2 v(t)^2} (\xi - L_g h)^2  \\&- \frac{k_2 \alpha_1}{2v(t)}\| x- \pi(u^\ast)\|^2 - \frac{L_\pi k_2 \alpha_1}{4v(t)} \tilde{u}^2\\
&-\frac{\alpha k}{4\beta_1 v(t)^2} (h(x)-h(x^\ast)) -\frac{\gamma k_{13}}{v(t)^2}  (\xi-L_g h)^2.
\end{align*}
This can be written as:
\begin{align*}
\dot{W}\leq&-\frac{\dot{v}}{v(t)}W+ \frac{\beta \dot{v}}{2 v(t)} \tilde{u}^2  - \frac{k_2 \alpha_1}{2v(t)}\| x- \pi(u^\ast)\|^2 \\& - \bigg(\frac{L_\pi k_2 \alpha_1}{4v(t)}-\frac{\alpha k \beta \gamma}{8 \beta_1 v(t)} \bigg) \tilde{u}^2-\frac{k}{4\beta_1 v(t)} W \\&-\bigg(\frac{\gamma k_{13}}{v(t)^2} -\frac{ \gamma \alpha k }{8 \beta_1 v(t)^2}\bigg) (\xi-L_g h)^2
\end{align*}
We note that $\frac{\dot{v}}{v(t)}=-\frac{2}{T-t}$. Choosing $k_{13}\geq \frac{\alpha k}{8 \beta_1}$ and $k_2\geq \frac{\alpha k \beta \gamma}{2 L_\pi \alpha_1}$, the last inequality leads to:
\begin{align*}
\dot{W}\leq&-\frac{\dot{v}}{v(t)}W -\frac{\alpha k}{4\beta_1 v(t)} W.
\end{align*}
Proceeding as in Lemma \ref{lem3}, we conclude that for $\frac{\alpha kT}{4\beta_1}>2$, the closed-loop system reaches optimal conditions in prescribed time $T$,  such that $\lim_{t\rightarrow T} x(t)=x^\ast$, $\lim_{t \rightarrow T} \hat{u}(t)=u^\ast$ and $\lim_{t\rightarrow T} \xi(t)=0$. This completes the proof.
\end{proof}

\subsection{Stability of the nominal system}

First, we consider the nominal closed-loop system:
\begin{equation}\label{eq:nomcl1}
	\begin{aligned}
		\dot{x} = & f(x) - k\bigg( 1+ \frac{T^2}{(T-t)^2}\bigg) g(x) \xi + g(x) \hat{u} + g(x) Aa\sin(\omega \tau)\\
		\dot{\hat{u}} = & -\frac{k}{\tau_I} \frac{T^2}{(T-t)^2}\xi \\
		\dot{\xi} = & -\omega_l \frac{T^2}{(T-t)^2} \bigg( \xi - \frac{2}{a} \sin(\omega \tau)\frac{T^2}{(T-t)^2}(-\omega_h^2 \eta+\omega_h y)\bigg)\\
		\dot{\eta} = & - \frac{T^2}{(T-t)^2} (\omega_h \eta - y(t)).
	\end{aligned}
\end{equation}
We define $\tilde{\eta}= \eta-y/\omega_h$. Using the notation defined above, we can rewrite the closed-loop dynamics as:
\begin{equation}\label{eq:nomcl2}
	\begin{aligned}
		\dot{x} = & f(x) - k \bigg( 1+ \frac{T^2}{(T-t)^2}\bigg) g(x) L_g h+ g(x) \hat{u} \\&- k  \bigg( 1+ \frac{T^2}{(T-t)^2}\bigg) g \tilde{\xi} +  g(x) A \sin(\omega \tau)\\
		\dot{\hat{u}} = & -\frac{k}{\tau_I}\frac{T^2}{(T-t)^2} L_g h - \frac{k}{\tau_I}\frac{T^2}{(T-t)^2} \tilde{\xi} \\
		\dot{\tilde{\xi}} = & -\omega_l\frac{T^2}{(T-t)^2} \tilde{\xi} - \frac{d L_gh}{dt} \\
		&-\omega_l\bigg(L_gh- \frac{2}{a} \sin(\omega \tau)\frac{T^2}{(T-t)^2} (-\omega_h^2 \tilde{\eta}) \bigg)\\
		\dot{\tilde{\eta}} = & -\omega_h \frac{T^2}{(T-t)^2} \tilde{\eta} -\frac{1}{\omega_h}\dot{y}.
	\end{aligned}
\end{equation}
Next, we consider the dynamics of $\tilde{\nu}$ in the $\tau$ timescale:
\begin{align*}
\frac{d\tilde{\nu}}{d\tau} = & -\omega_h\tilde{\nu} -\frac{d^2 y}{d\tau^2}
\end{align*}
Since we have, $\tilde{\nu}=-\omega_h^2\tilde{\eta} - \frac{dy}{d\tau}$, then one can write:
\begin{align*}
		\dot{x} = & f(x) - k \bigg( 1+ \frac{T^2}{(T-t)^2}\bigg) g(x) L_g h+ g(x) \hat{u} \\&- k  \bigg( 1+ \frac{T^2}{(T-t)^2}\bigg) g \tilde{\xi} +  g(x) A \sin(\omega \tau)\\
		\dot{\hat{u}} = & -\frac{k}{\tau_I}\frac{T^2}{(T-t)^2} L_g h - \frac{k}{\tau_I}\frac{T^2}{(T-t)^2} \tilde{\xi} \\
		\dot{\tilde{\xi}} = & -\omega_l\frac{T^2}{(T-t)^2} \tilde{\xi} - \frac{d L_gh}{dt} \\
		&-\omega_l\bigg(L_gh- \frac{2}{a} \sin(\omega \tau)\frac{T^2}{(T-t)^2} (\tilde{\nu}+\frac{dy}{d\tau}) \bigg)\\
		\dot{\tilde{\nu}} = & -\omega_h \frac{T^2}{(T-t)^2} \tilde{\nu} -\frac{T^2}{(T-t)^2}\frac{d^2 y}{d\tau^2}
\end{align*}
We notice that:
\begin{align*}
\frac{T^2}{(T-t)^2}\frac{d^2 y}{d\tau^2}= \frac{d^2 y}{dt d\tau}=\frac{d^2y}{dt^2} \frac{dt}{d\tau}.
\end{align*}
This yields:
\begin{align*}
		\dot{x} = & f(x) - k \bigg( 1+ \frac{T^2}{(T-t)^2}\bigg) g(x) L_g h+ g(x) \hat{u} \\&- k  \bigg( 1+ \frac{T^2}{(T-t)^2}\bigg) g \tilde{\xi} +  g(x) A \sin(\omega \tau)\\
		\dot{\hat{u}} = & -\frac{k}{\tau_I}\frac{T^2}{(T-t)^2} L_g h - \frac{k}{\tau_I}\frac{T^2}{(T-t)^2} \tilde{\xi} \\
		\dot{\tilde{\xi}} = & -\omega_l\frac{T^2}{(T-t)^2} \tilde{\xi} - \frac{d L_gh}{dt} -\omega_l\frac{T^2}{(T-t)^2} \bigg(L_gh\\
		&\hspace{0.25in}-\frac{2}{a}\sin(\omega_h \tau) \dot{y}- \frac{2}{a} \sin(\omega \tau)\frac{T^2}{(T-t)^2} \tilde{\nu} \bigg)\\
		\dot{\tilde{\nu}} = & -\omega_h \frac{T^2}{(T-t)^2} \tilde{\nu} - \frac{(T-t)^2}{T^2}\ddot{y}
\end{align*}
Using a standard argument, we let $\omega_h\rightarrow \infty$ to generate the slow dynamics:
\begin{align*}
		\dot{x} = & f(x) - k \bigg( 1+ \frac{T^2}{(T-t)^2}\bigg) g(x) L_g h+ g(x) \hat{u} \\&- k  \bigg( 1+ \frac{T^2}{(T-t)^2}\bigg) g \tilde{\xi} +  g(x) A \sin(\omega \tau)\\
		\dot{\hat{u}} = & -\frac{k}{\tau_I}\frac{T^2}{(T-t)^2} L_g h - \frac{k}{\tau_I}\frac{T^2}{(T-t)^2} \tilde{\xi} \\
		\dot{\tilde{\xi}} = & -\omega_l\frac{T^2}{(T-t)^2} \tilde{\xi} - \frac{d L_gh}{dt} -\omega_l\frac{T^2}{(T-t)^2} \bigg(L_gh\\
		&\hspace{0.25in}-\frac{2}{a}\sin(\omega_h \tau) \dot{y}\bigg)
\end{align*}

Letting $K(t) = k \bigg( 1+ \frac{T^2}{(T-t)^2}\bigg)$, $\tilde{u}=\hat{u}-u^\ast$ and $\tilde{x}=x-x^\ast$, \eqref{eq:nomcl1} can be rewritten as:
\begin{align*}
		\dot{\tilde{x}} = & f(x) - K(t) g(x) L_g h+ g(x) \hat{u} - K(t) g \tilde{\xi} +  g(x) A \sin(\omega \tau)\\
		\dot{\tilde{u}} = & -\frac{k}{\tau_I}\frac{T^2}{(T-t)^2} L_g h - \frac{k}{\tau_I}\frac{T^2}{(T-t)^2} \tilde{\xi} \\
		\dot{\tilde{\xi}} = & -\omega_l\frac{T^2}{(T-t)^2} \tilde{\xi} -\omega_l\frac{T^2}{(T-t)^2} \bigg(L_gh\\
		&\hspace{0.75in}-\frac{2}{a}\sin(\omega_h \tau)  \dot{y} \bigg)- \frac{d L_gh}{dt}
\end{align*}
or,
\begin{align*}
		\dot{\tilde{x}} = & f(x) - K(t) g(x) L_g h+ g(x) \hat{u} - K(t) g \tilde{\xi} +  g(x) A \sin(\omega \tau)\\
		\dot{\tilde{u}} = & -\frac{k}{\tau_I}\frac{T^2}{(T-t)^2} L_g h - \frac{k}{\tau_I}\frac{T^2}{(T-t)^2} \tilde{\xi} \\
		\dot{\tilde{\xi}} = & -\omega_l\frac{T^2}{(T-t)^2} \bigg(\tilde{\xi} +(1-2 sin^2(\omega \tau)) L_g h \\
		&-\frac{2}{a}\sin(\omega \tau)(L_f h + L_g h \hat{u} - K(t) (L_gh)^2+K(t) L_g h\tilde{\xi}\bigg) \\
		&- (L_f L_g h + L_g^2 h \hat{u} - K(t) L_g^2 h L_g h + K(t) L_g^2 h \tilde{\xi} \\ & \hspace{0.5in} + L_g^2h a \sin(\omega t))
\end{align*}
This leads to:
\begin{align*}
		\dot{\tilde{x}} = & f(x) - K(t) g(x) L_g h+ g(x) \hat{u} - K(t) g \tilde{\xi} +  g(x) A \sin(\omega \tau)\\
		\dot{\tilde{u}} = & -\frac{k}{\tau_I}\frac{T^2}{(T-t)^2} L_g h - \frac{k}{\tau_I}\frac{T^2}{(T-t)^2} \tilde{\xi} \\
		\dot{\tilde{\xi}} = & -\omega_l\frac{T^2}{(T-t)^2} \tilde{\xi} + \frac{T^2}{(T-t)^2} B_\xi(\tau,\tilde{x},\tilde{u},\tilde{\xi})  \\
		&- (L_f L_g h + L_g^2 h \hat{u} - K(t) L_g^2 h L_g h + K(t) L_g^2 h \tilde{\xi}).
\end{align*}
where 
\begin{align*} B_\xi&(\tau,\tilde{x},\tilde{u},\tilde{\xi})=(1-2 \sin^2(\omega \tau )) L_g h- \frac{2}{a} \sin(\omega \tau ) (L_fh \\& + L_gh \hat{u}- K(t) (L_g h)^2 + K(t) L_gh \tilde{\xi} - \frac{a^2}{2} L_g^2h  ).
\end{align*}

We let $\bm{X}=\begin{bmatrix} \tilde{x}^T &  \tilde{u}^T & \tilde{\xi}^T \end{bmatrix}^T $ and write the dynamics of the nominal system as:
\begin{align} \label{eq:nomsys}
\dot{\bm{X}} = F(\bm{X}) + \bm{B}(t,\bm{X})
\end{align}
where $F(\bm{X}) = \begin{bmatrix} F_1(\bm(X)) & F_2(\bm{X}) & F_3(\bm{X}) \end{bmatrix}^T$ with
\begin{align*}
F_1(\bm{X}) =&  f(x) - K(t) g(x) L_g h+ g(x)u^\ast + g(x) \tilde{u} - K(t) g(x) \tilde{\xi} \\
F_2(\bm{X}) =& -\frac{k}{\tau_I}\frac{T^2}{(T-t)^2} L_g h - \frac{k}{\tau_I}\frac{T^2}{(T-t)^2} \tilde{\xi} \\
F_3(\bm{X}) =&  -\omega_l\frac{T^2}{(T-t)^2} \tilde{\xi}  \\
		&- (L_f L_g h + L_g^2 h \hat{u} - K(t) L_g^2 h L_g h + K(t) L_g^2 h \tilde{\xi})
\end{align*}
and
$$\bm{B}(t,\bm{X}) = \begin{bmatrix} g(x) a \sin(\omega \tau) \\ 0 \\ \frac{T^2}{(T-t)^2}  B_\xi(\tau,\tilde{x},\tilde{u},\tilde{\xi})\end{bmatrix}.$$

First, we note that by the analysis in the brief section, there is a positive definite function $V(\bm{X})$ such that $$\frac{\partial V}{\partial \bm{X}} F(\bm{X}) \leq \frac{2}{T-t}V-\frac{\alpha k_1 T^2}{\beta (T-t)^2} V.$$ Thus, it follows that if $\frac{\alpha k_1 T}{\beta} = 2 + k^\ast T$ for some positive constant $k^\ast$ then we obtain:
$$\frac{\partial V}{\partial \bm{X}} F(\bm{X}) \leq -\frac{k^\ast T^2}{ (T-t)^2} V(\bm{X}).$$ 

We then consider the trajectories of the system in the $\tau$ timescale. This yields the system:
\begin{align*}
\frac{d\bm{X}}{d\tau} = \bar{F}(\bm{X})+\bar{\bm{B}}(\tau,\bm{X})
\end{align*}
where $\bar{F}(\bm{X})=\frac{T^2}{(T+\tau)^2} F(\bm{X})$ and $$\bar{\bm{B}}(\tau,\bm{X}) = \begin{bmatrix} \frac{T^2}{(T+\tau)^2} g(x) a \sin(\omega \tau) \\ 0 \\ B_\xi(\tau,\tilde{x},\tilde{u},\tilde{\xi})\end{bmatrix}.$$ 

Next, we note that the functions $B_1(\tau,\bm{X})=g(x)a\sin(\omega \tau)$ and $B_\xi(\tau, \tilde{x},\tilde{u}, \tilde{\xi})$ is $T$-periodic, locally Lipschitz in $\bm{X}$ with Lipschitz constant $L$ uniformly in $t$ and $|B_1(t,0)|= |g(x^\ast) a \sin(\omega \tau)| \leq c$ and $ | B_\xi(\tau,0,0,0) | \leq c$  for a positive constant $c$.

As in \cite{AEYELS19991091}, \cite{PEUTEMAN2011192} and \cite{teel1999semi}, let $L$ denote the Lipschitz constant for $\bar{B}(\tau,\bm{X})$, $F(\bm{X})$ and $\frac{\partial V}{\partial \bm{X}}$ over some compact set $D$. Since $\frac{T^2}{(T+\tau)^2} \leq 1$ , $\forall \tau >0$,  we can write:
\begin{align} 
\bigg| \frac{\partial V}{\partial \bm{X}}(\bm{X}_1)\bar{\bm{B}}(\tau,\bm{X}_1)-&\frac{\partial V}{\partial \bm{X}}(\bm{X}_2)\bar{\bm{B}}(\tau,\bm{X}_2) \bigg|
\leq \nonumber \\& (L\|\bm{X}_1\| + L \|\bm{X}_2\| + c) L \| \bm{X}_1-\bm{X}_2\|. \label{eq:dvdxL}
\end{align}
for any $\bm{X}_1$, $\bm{X}_2  \in D$. 
For the differential equation, one obtains:
\begin{align*} 
\| \bm{X}(\tau)-\bm{X}(\tau_0)\| \leq& \int_{\tau_0}^{\tau} \bigg(\|\bar{F}(\bm{X}(\sigma))-\bar{F}(\bm{X}(\tau_0))\| \\
&+\| \bm{B}(\sigma, \bm{X}(\sigma)-\bm{B}(\sigma,\bm{X}(\tau_0)\| \\
&+\| \bm{B}(\sigma, \bm{X}(\tau_0)-\bm{B}(\sigma,0)\| \\
&+ \|\bar{F}(\bm{X}(\tau_0))\| + \|\bm{B}(\sigma,0)\|\bigg)d\sigma
\end{align*}
which gives:
\begin{align*} 
\| \bm{X}(\tau)-\bm{X}(\tau_0)\| \leq \int_{\tau_0}^{\tau} (2L\|&\bm{X}(\sigma)-\bm{X}(\tau_0)\| \\
&+2L\|\bm{X}(\tau_0)\|  +c)d\sigma.
\end{align*}
By the Gronwall lemma, we get:
\begin{align*}
\| \bm{X}(\tau)-\bm{X}(\tau_0)\|  \leq \left( \|\bm{X}(\tau_0)\| + \frac{c}{2L} \right) (e^{2L(\tau-\tau_0)} - 1)
\end{align*}
In same way, we can show that:
\begin{align*}
\| \bm{X}(\tau)\| \leq  \|\bm{X}(\tau_0)\|e^{2L(\tau-\tau_0)} + \frac{c}{2L}  (e^{2L(\tau-\tau_0)} - 1).
\end{align*}

We consider a ball $B_\varepsilon$ where $\varepsilon>0$ is chosen such that $B_\sigma \subset D$. Let $
\sigma < \varepsilon$. Based on the last statement, there is a positive constant $\delta>0$ small enough such that 
\begin{align*}
 \| \bm{X}(\tau_0) \|^2 \leq\frac{\alpha_2}{\alpha_1}\sigma  \Rightarrow  \|\bm{X}(\tau) \|^2 \leq 2\frac{\alpha_2}{\alpha_1} \sigma 
\end{align*}
$\forall \tau \in [\tau_0,\, \tau_0+\delta]$.

We consider an increasing sequence of times $\tau_k$ for $k \in \mathbb{N}$ such that $\tau_{k+1}-\tau_{k}=\delta$ and we compute the change of the Lyapunov function $V$ at each time $\tau_k$. This yields:
\begin{align*}
V(\bm{X}&(\tau_{k+1}))-V(\bm{X}(\tau_k)) \\
& =\int_{\tau_k}^{\tau_{k+1}} \frac{\partial V}{\partial \bm{X}}(\bm{X}(\sigma))(\bar{F}(\bm{X}(\sigma)+\bar{\bm{B}}(\sigma,\bm{X}(\sigma)) d\sigma \\
& \leq - k^\ast \int_{\tau_k}^{\tau_{k+1}}  \| \bm{X}(\sigma)\|^2 d\sigma \\ 
&\hspace{0.5in} + \int_{\tau_k}^{\tau_{k+1}} \bigg(\frac{\partial V}{\partial \bm{X}}(\bm{X}(\sigma))\bar{\bm{B}}(\sigma,\bm{X}(\sigma))\bigg) d\sigma
\end{align*} 
This can be rewritten as:
\begin{align*}
V(\bm{X}&(\tau_{k+1}))-V(\bm{X}(\tau_k)) \\
& =\int_{\tau_k}^{\tau_{k+1}} \frac{\partial V}{\partial \bm{X}}(\bm{X}(\sigma))(\bar{F}(\bm{X}(\sigma)+\bar{\bm{B}}(\sigma,\bm{X}(\sigma)) d\sigma \\
& \leq - k^\ast \int_{\tau_k}^{\tau_{k+1}}  \| \bm{X}(\sigma)\|^2 d\sigma \\ 
&\hspace{0.25in} + \int_{\tau_k}^{\tau_{k+1}} \bigg(\frac{\partial V}{\partial \bm{X}}(\bm{X}(\sigma))\bar{\bm{B}}(\sigma,\bm{X}(\sigma)) \\
&\hspace{0.25in} -\frac{\partial V}{\partial \bm{X}}(\bm{X}(\tau_k))\bar{\bm{B}}(\tau_k,\bm{X}(\tau_k))\bigg) d\sigma \\
& +\int_{\tau_k}^{\tau_{k+1}} \frac{\partial V}{\partial \bm{X}}(\bm{X}(\tau_k))\bar{\bm{B}}(\tau_k,\bm{X}(\tau_k)) d\sigma
\end{align*} 
or,
\begin{align*}
V(\bm{X}&(\tau_{k+1}))-V(\bm{X}(\tau_k)) \\
& =\int_{\tau_k}^{\tau_{k+1}} \frac{\partial V}{\partial \bm{X}}(\bm{X}(\sigma))(\bar{F}(\bm{X}(\sigma)+\bar{\bm{B}}(\sigma,\bm{X}(\sigma)) d\sigma \\
& \leq - k^\ast \int_{\tau_k}^{\tau_{k+1}}  \| \bm{X}(\sigma)\|^2 d\sigma \\ 
&\hspace{0.25in} + \int_{\tau_k}^{\tau_{k+1}} \bigg(\frac{\partial V}{\partial \bm{X}}(\bm{X}(\sigma))\bar{\bm{B}}(\sigma,\bm{X}(\sigma)) \\
&\hspace{0.25in} -\frac{\partial V}{\partial \bm{X}}(\bm{X}(\tau_k))\bar{\bm{B}}(\tau_k,\bm{X}(\tau_k))\bigg) d\sigma \\
& +\int_{\tau_k}^{\tau_{k+1}} \bigg( \frac{\partial V}{\partial \bm{X}}(\bm{X}(\tau_k))\bar{\bm{B}}(\tau_k,\bm{X}(\tau_k)) \\
&\hspace{0.25in} -\frac{\partial V}{\partial \bm{X}}(\bm{X}(\tau_k))\bar{\bm{B}}(\tau_k,0) \bigg) d\sigma \\
&+\int_{\tau_k}^{\tau_{k+1}} \frac{\partial V}{\partial \bm{X}}(\bm{X}(\tau_k))\bar{\bm{B}}(\tau_k,0)  d\sigma
\end{align*}
Following the inequality \eqref{eq:dvdxL}, this can be upper bounded by:
\begin{align*}
V&(\bm{X}(\tau_{k+1}))-V(\bm{X}(\tau_k)) \\
& \leq - k^\ast \int_{\tau_k}^{\tau_{k+1}}  \| \bm{X}(\sigma)\|^2 d\sigma \\ 
& + \int_{\tau_k}^{\tau_{k+1}} (L\|\bm{X}(\sigma)\|+ L\|\bm{X}(\tau_k)\| + c ) \|  \bm{X}(\sigma)-\bm{X}(\tau_k)\| d\sigma \\
&\hspace{0.25in} +(\tau_{k+1}-\tau_{k}) L^2 \|\bm{X}(\tau_k)\|^2 +(\tau_{k+1}-\tau_{k}) cL \|\bm{X}(\tau_k)\|. 
\end{align*}
This yields:
\begin{align*}
V&(\bm{X}(\tau_{k+1}))-V(\bm{X}(\tau_k)) \\
& \leq - k^\ast \int_{\tau_k}^{\tau_{k+1}}  \| \bm{X}(\sigma)\|^2 d\sigma \\ 
& + \int_{\tau_k}^{\tau_{k+1}} (L\|\bm{X}(\tau_k)\| (e^{2L(\tau-\tau_k)}+1) +  \frac{c}{2L}  (e^{2L(\tau-\tau_k)} - 1) \\
& \hspace{0.25in}  + c ) \left( \|\bm{X}(\tau_k)\| + \frac{c}{2L} \right) (e^{2L(\tau-\tau_{k})} - 1) d\sigma \\
&\hspace{0.25in} +(\tau_{k+1}-\tau_{k}) L^2 \|\bm{X}(\tau_k)\|^2 +(\tau_{k+1}-\tau_{k}) cL \|\bm{X}(\tau_k)\| 
\end{align*}
or, 
\begin{align*}
V&(\bm{X}(\tau_{k+1}))-V(\bm{X}(\tau_k)) \\
& \leq - k^\ast \int_{\tau_k}^{\tau_{k+1}}  \| \bm{X}(\sigma)\|^2 d\sigma \\ 
& + \int_{\tau_k}^{\tau_{k+1}} \bigg(L\|\bm{X}(\tau_k)\|^2 (e^{2L(\tau-\tau_k)}+1)(e^{2L(\tau_k-\tau_0)} - 1) \\
&\hspace{0.25in}+  \frac{c}{2} \|\bm{X}(\tau_k)\| (e^{2L(\tau_k-\tau_0)} + 1)(e^{2L(\tau_k-\tau_0)} - 1) \\
& \hspace{0.25in}  + \frac{c}{2L}  \|\bm{X}(\tau_k)\|(e^{2L(\tau-\tau_{k})} - 1+2L)  (e^{2L(\tau_k-\tau_0)} - 1) \\
& \hspace{0.25in} +\frac{c^2}{4L^2}(e^{2L(\tau_k-\tau_0)} - 1)^2 + \frac{c^2}{2L}  (e^{2L(\tau_k-\tau_0)} - 1) \bigg) d\sigma \\
&\hspace{0.25in} +(\tau_{k+1}-\tau_{k}) L^2 \|\bm{X}(\tau_k)\|^2 +(\tau_{k+1}-\tau_{k}) cL \|\bm{X}(\tau_k)\|. 
\end{align*}
Upon integration and substitution of $\tau_{k+1}-\tau_k= 2 \pi \epsilon$ with $\epsilon=\omega^{-1}$, we obtain: 
\begin{align*}
V(\bm{X}(\tau_{k+1}))-&V(\bm{X}(\tau_k))  \leq - k^\ast  \| \bm{X}(\tau_k)\|^2 +\rho_0(\epsilon)\\
&+ \rho_1(\epsilon) \| \bm{X}(\tau_k)\|  + \rho_2(\epsilon) \| \bm{X}(\tau_k)\|^2
\end{align*}
where 
\begin{align*}
\rho_0(\epsilon)= &\frac{c^2}{4L^2} \bigg(\frac{1}{4L} (e^{8L\pi \epsilon}-4e^{4L\pi \epsilon}+3)+2\pi \epsilon\bigg) \\
&+ \frac{c^2}{2L^2} \bigg(\frac{1}{2L}( e^{4L\pi \epsilon} -1) -2\pi \epsilon\bigg)\\
\rho_1(\epsilon)= &  \bigg(\frac{c}{2}+ \frac{c}{2L} \bigg) \bigg(\frac{1}{4L} (e^{8L\pi\epsilon}-1)-2\pi \epsilon\bigg) \\
&+c \bigg(\frac{1}{2L} (e^{4L\pi \epsilon}-1)-2\pi \epsilon\bigg)+2\pi \epsilon cL\\
\rho_2(\epsilon) = & L\bigg(\frac{1}{4L} (e^{8L\pi\epsilon)}-1)-2\pi \epsilon\bigg) + 2\pi \epsilon L^2. 
\end{align*}
We let $k^\ast=k_1 + \rho_2(\epsilon)$. This yields:
\begin{align*}
V(\bm{X}(\tau_{k+1}))-V(\bm{X}(\tau_k))  \leq -& k_1  \| \bm{X}(\tau_k)\|^2 +\rho_0(\epsilon)\\
&+ \rho_1(\epsilon) \| \bm{X}(\tau_k)\|
\end{align*}
For $c\in(0,1)$, we can write:
\begin{align*}
V(\bm{X}(\tau_{k+1}))-&V(\bm{X}(\tau_k))  \leq - k_1(1-c)  \| \bm{X}(\tau_k)\|^2 +\rho_0(\epsilon)\\
&-k_1 c \| \bm{X}(\tau_k)\| \bigg(\|\bm{X}(\tau_k)\|-\frac{\rho_1(\epsilon)}{k_1 c}\bigg).
\end{align*}
We can then conclude that:
\begin{align*}
V(\bm{X}(\tau_{k+1}))-V(\bm{X}(\tau_k))  \leq -& k_1(1-c)  \| \bm{X}(\tau_k)\|^2 +\rho_0(\epsilon)
\end{align*}
for all $\bm{X}(\tau_k)$ such that 
$\|\bm{X}(\tau_k)\| \geq \frac{\rho_1(\epsilon)}{k_1 c}$.  Similarly, it follows that:
\begin{align*}
V(\bm{X}(\tau_{k+1}))-V(\bm{X}(\tau_k))  \leq -& \frac{k_1(1-c)}{2}  \| \bm{X}(\tau_k)\|^2
\end{align*}
for $\| \bm{X}(\tau_k)\|^2 \geq \max \left[\frac{2 \rho_0(\epsilon)}{k_1 (1-c)},\, \frac{\rho_1(\epsilon)^2}{k_1^2 c^2}\right]=M_1(\epsilon)$.
We let $\nu_1=\frac{k_1(1-c)}{2}$ and write:
\begin{align*}
V(\bm{X}(\tau_{1}))-V(\bm{X}(\tau_0))  \leq -&\nu_1 \| \bm{X}(\tau_0)\|^2
\end{align*}
for $\|\bm{X}(\tau_0)\|^2 \geq M_1(\epsilon)$. We let $\epsilon$ be small enough that $M_1(\epsilon)\leq \frac{\alpha_2}{\alpha_1}\sigma$.
This yields:
\begin{align*}
V(\bm{X}(\tau_{1})) \leq \min\left[\bigg(1-\frac{\nu_1}{\alpha_2}\bigg) V(\bm{X}(\tau_0)),\,\, \alpha_1 M_1(\epsilon)\right] .
\end{align*}
where $0\leq \bigg(1-\frac{\nu_1}{\alpha_2}\bigg)<1$. 
At the $n^{th}$ step, one obtains:
\begin{align*}
V(\bm{X}(\tau_{n})) \leq \min\bigg[\bigg(&1-\frac{\nu_1}{\alpha_2}\bigg)^n V(\bm{X}(\tau_{0}) ),\,\, \alpha_1 M_1(\epsilon)\bigg].
\end{align*}
This yields:
\begin{align*}
\|\bm{X}(\tau_{n})\|^2 \leq \min\bigg[\bigg(&1-\frac{\nu_1}{\alpha_2}\bigg)^n \frac{\alpha_2}{\alpha_1} \|\bm{X}(\tau_{0})\|^2,\,\, M_1(\epsilon)\bigg]. 
\end{align*}
Thus, we see that $\|\bm{X}(\tau_{n})\|^2 \leq \frac{\alpha_2}{\alpha_1} \sigma$ $\forall n \in \mathbb{N}$.
Let $n^\ast$ be such that:
\begin{align*}
\nu_1^\ast = \bigg(1-\frac{\nu_1}{\alpha_2}\bigg)^{n^{\ast}} \frac{\alpha_2}{\alpha_1} <1.
\end{align*}
Thus, we obtain:
\begin{align*}
\|\bm{X}(\tau_{n})\| \leq \min\bigg[\sqrt{\nu_1^\ast} \|\bm{X}(\tau_{0})\|,\,\,  \sqrt{M_1(\epsilon)}\bigg].
\end{align*}
As a result, we can see that $\|\bm{X}(\tau_{n^{\ast}})\|\leq \frac{\alpha_2}{\alpha_1} \sigma$.  Furthermore, we let $\lambda=\frac{\ln(1/\nu_1^\ast)}{4 n^\ast \pi \epsilon}$ to obtain:
 \begin{align*}
\|\bm{X}(\tau_{n^{\ast}})\| \leq \min\bigg[ e^{-\lambda(\tau_{n^\ast}-\tau_0)} \|\bm{X}(\tau_{0})\|,\,\,  \sqrt{M_1(\epsilon)}\bigg].
\end{align*}
Thus, it follows that one can repeat the procedure for any multiple of $n^\ast$ such that:
 \begin{align*}
\|\bm{X}(\tau_{mn^{\ast}})\| \leq \min\bigg[ e^{-\lambda(\tau_{mn^\ast}-\tau_0)} \|\bm{X}(\tau_{0})\|,\,\,\sqrt{M_1(\epsilon)}\bigg]
\end{align*}
for any $m\in \mathbb{N}$. 

Since $\|\bm{X}(\tau_{n})\| \leq \frac{\alpha_2}{\alpha_1} \sigma$ $\forall n$, it follows that the solution $\bm{X}(\tau)$ is such that:
\begin{align*} 
\| \bm{X}(\tau) \| \leq e^{2L(\tau-\tau_{mn^{\ast}})} \|\bm{X}(\tau_{mn^{\ast}})\| +\frac{c}{2L}\bigg( e^{2L(\tau-\tau_{mn^{\ast}})}-1\bigg)
\end{align*}
for $\tau \in [\tau_{mn^{\ast}},\,\tau_{(m+1)n^{\ast}}[$. As a result, we obtain:
\begin{align*} 
\| \bm{X}(\tau) \| &\leq \\&e^{2L(\tau-\tau_{mn^{\ast}})} \min\bigg[e^{-\lambda(\tau_{mn^\ast}-\tau_0)} \|\bm{X}(\tau_{0})\|,\,\, \sqrt{M_1(\epsilon)}\bigg] \\
& +\frac{c}{2L}\bigg( e^{2L(\tau-\tau_{mn^{\ast}})}-1\bigg) 
\end{align*}
Rearranging, this yields:
\begin{align*} 
\| \bm{X}(\tau) \| \leq & e^{(\lambda + 2L)(\tau-\tau_{mn^{\ast}})} e^{-\lambda(\tau-\tau_0)} \|\bm{X}(\tau_{0})\|\\
&+ e^{2L(\tau-\tau_{mn^{\ast}})} \sqrt{M_1(\epsilon)}+\frac{c}{2L}\bigg( e^{2L(\tau-\tau_{mn^{\ast}})}-1\bigg) 
\end{align*}
or
\begin{align*} 
\| \bm{X}(\tau) \| \leq e^{(\lambda + 2L)n^{\ast}\delta} &e^{-\lambda(\tau-\tau_0)} \|\bm{X}(\tau_{0})\| \\
&+ e^{2Ln^{\ast}\delta} \sqrt{M_1(\epsilon)}+\frac{c}{2L}\bigg( e^{2Ln^{\ast}\delta}-1\bigg).
\end{align*}
If we let $\delta = 2\pi \epsilon$, this gives:
\begin{align*} 
\| \bm{X}(\tau) \| \leq &e^{(\lambda + 2L)2\pi n^{\ast}\epsilon} e^{-\lambda(\tau-\tau_0)} \|\bm{X}(\tau_{0})\| \\
&+  e^{4 \pi L n^{\ast}\epsilon}  \sqrt{M_1(\epsilon)}+\frac{c}{2L}\bigg( e^{4\pi Ln^{\ast}\epsilon}-1\bigg)
\end{align*}
which holds for all $\tau\geq\tau_0$ and all $\|\bm{X}(\tau_0) \|^2 \leq \frac{\alpha_2}{\alpha_1} \sigma$. As a result, practical stability is achieved in the $\tau$ timescale.
Since practical stability is achieved as $\tau\rightarrow \infty$ then we conclude that prescribed-time practical stability of the closed-loop nominal system is achieved as $t\rightarrow T$. 

%
%
%


\section{Simulation Study} \label{sec4}

\subsection{General Nonlinear System}
We consider the following nonlinear system:
\begin{align*}
   \dot{x}_1 &= (-x_1 + x_2^2) \\
   \dot{x}_2 &= -x_1 + x_2 + u \\
   y &= 1 + x_2^2 + x_1^2.
\end{align*}
The objective is to stabilize the system and minimize the cost function within the prescribed time. The control system with the proposed timescale transformation is implemented with the following tuning parameters and initial conditions: $T = 5, A = 25, \omega_1 = 150, \omega_h = 2000, \omega_l = 3, k = 25, \tau_I = 0.5, x_1(0) = 1, x_2(0) = 2$.

The closed-loop trajectories of the state variables and the output are shown in Figure \ref{fig:gen_sys_states}. It can be seen that the system states and output converges to the optimum within the prescribed time. The estimate of the output is shown in Figure \ref{fig:gen_sys_out} and can be seen to be driven to converge to zero. 

\begin{figure}[h]
   \centering
   \includegraphics[width = 0.75\linewidth]{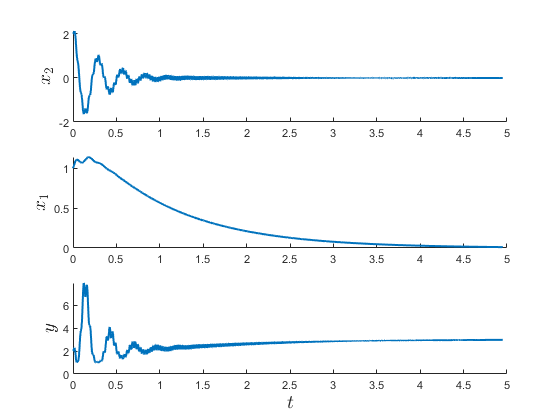}
   \caption{Closed-loop trajectories of the system. The top plot shows the state $x_2$, the middle plot shows $x_1$ and the bottom plot shows the output $y$.}
   \label{fig:gen_sys_states}
\end{figure}

\begin{figure}[h]
   \centering
   \includegraphics[width = 0.75\linewidth]{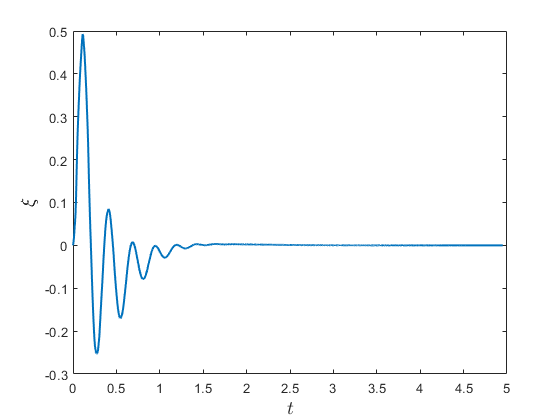}
   \caption{Estimate of the output $\xi$.}
   \label{fig:gen_sys_out}
\end{figure}

\subsection{Batch-fed Bioreactor}
Now we consider a batch-fed bioreactor described by the following system of equations:
\begin{align*}
    \dot{x}_1 &= \frac{x_1 x_2}{0.2 + x_2} - x_1u \\
    \dot{x}_2 &= -\frac{2 x_1 x_2}{0.2 + x_2} + (10 - x_2)u \\
    y &= -\frac{x_1 x_2}{0.2 + x_2}
\end{align*}
where the states $x_1$ and $x_2$ are the biomass and substrate concentrations respectively, and $u$ is the feed rate. In this instance the cost function is optimizing the rate of generation of biomass concentration. 
The system parameters used in this simulation are: $T = 50, A = 0.5, \omega_1 = 150, \omega_h = 2000, \omega_l = 3, k = 2, \tau_I = 0.5, x_1(0) = 1, x_2(0) = 1$.

Figure \ref{fig:bio_states} presents the states and output results of the bioreactor, with the targeted state $x_2$ seen to be converging at the optimum within the prescribed time. The estimate of the output is also shown to converge to zero in Figure \ref{fig:bio_out}.
\begin{figure}[htb]
    \centering
    \includegraphics[width = 0.75\linewidth]{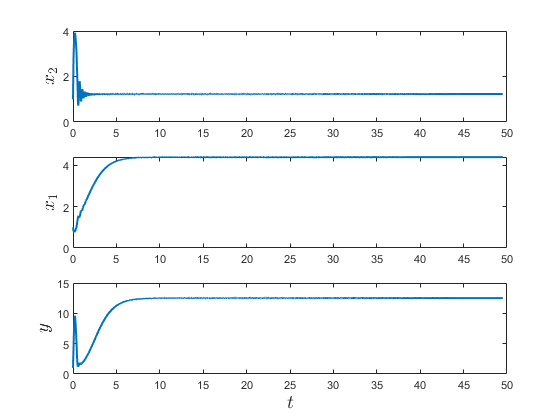}
    \caption{Closed-loop trajectories of the bioreactor. The top plot shows the state $x_2$, the middle plot shows $x_1$ and the bottom plot shows the output $y$.}
    \label{fig:bio_states}
\end{figure}

\begin{figure}[htb]
    \centering
    \includegraphics[width = 0.75\linewidth]{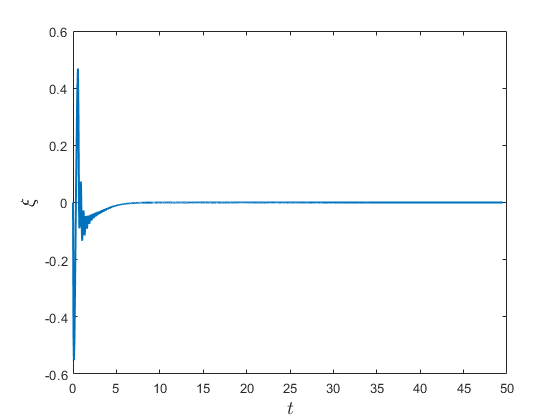}
    \caption{Estimate of the output $\xi$ in the bioreactor.}
    \label{fig:bio_out}
\end{figure}



\begin{figure}[htb]
    \centering
    \includegraphics[width = 0.75\linewidth]{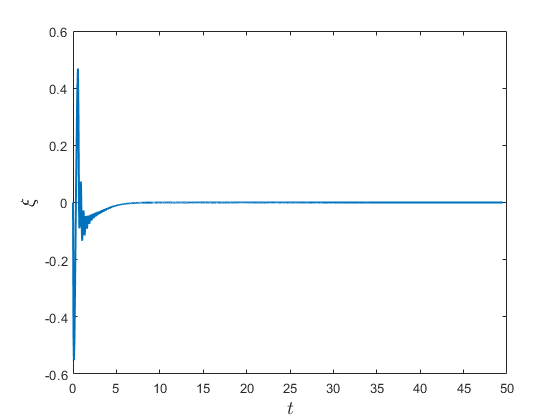}
    \caption{Estimate of the output $\xi$ in the bioreactor.}
    \label{fig:bio_out}
\end{figure}

\section{Conclusion} \label{sec5}
In this paper a new strategy is proposed for providing prescribed time convergence in unknown dynamical nonlinear systems using a dual-mode extremum-seeking control design. The approach expresses the system in a new timescale in which a state-feedback controller is implemented to achieve exponential stability. Through the use of the timescale transformation, the system is then stabilized in prescribed time in the original timescale to the optimum steady-state for the state, input, and output variables. In future work, we will look to extend the control to continue beyond the prescribed settling time.

\bibliography{DMPTESC}
\end{document}